\documentclass{amsart}
\usepackage[utf8]{inputenc}
\usepackage{articletemplate}
\usepackage{todonotes}

\usepackage{xcolor}
\usepackage{comment}
\newcommand{\br}{\mathbf{r}}

\newcommand{\Z}{\mathbb{Z}}
\newcommand{\R}{\mathbb{R}}
\newcommand{\C}{\mathbb{C}}

\newcommand{\bx}{\textbf{x}}
\newcommand{\ba}{\textbf{a}}

\newcommand{\bn}{\textbf{n}}

\newcommand{\vol}{\textrm{vol}}

\newcommand{\lcm}{\textrm{lcm}}

\usepackage{xcolor}

\title{Small scale distribution of linear patterns of primes}
\author[M. Pandey]{Mayank Pandey}
\address{Department of Mathematics, Princeton University, Princeton, NJ 08540}
\email{mayankpandey9973@gmail.com}

\author[K. Woo]{Katharine Woo}
\address{Department of Mathematics, Princeton University, Princeton, NJ 08540}
\email{khwoo@princeton.edu}

\begin{document}
\begin{abstract}
    Let $\Psi$ be a system of linear forms with finite complexity. In their seminal paper, Green and Tao showed the following prime number theorem for values of the system $\Psi$: $$\sum_{x\in [-N,N]^d} \prod_{i=1}^t \charf{\mathcal{P}}(\psi_i(x)) \sim \frac{(2N)^d}{(\log N)^t} \prod_{p} \beta_p,$$
    where $\beta_p$ are the corresponding local densities. In this paper, we demonstrate limits to equidistribution of these primes on small scales; we show the analog to Maier's result on primes in short intervals. In particular, we show that for all $\lambda > 1$, there exist $\delta_\lambda^\pm > 0$ such that for $N$ sufficiently large, there exist boxes $B^\pm\subset [-N, N]^d$ of sidelengths at least 
    $(\log N)^\lambda$ such that
    $$\sum_{x\in B^+} \prod_{i=1}^t \charf{\mathcal{P}}(\psi_i(x)) > (1+\delta^+) 
    \frac{\vol(B^+)}{(\log N)^t} \prod_{p}\beta_p,$$
    $$\sum_{x\in B^-} \prod_{i=1}^t \charf{\mathcal{P}}(\psi_i(x)) < (1-\delta^-) 
    \frac{\vol(B^-)}{(\log N)^t} \prod_{p}\beta_p.$$
\end{abstract}

\maketitle
\section{Introduction}
In 1936, Cram\'er introduced a probabilistic model for predicting features of the primes. 
In this model, $\charf{n\text{ is prime}}$ is modeled by 
a random variable $X_n$ with $\PP(X_n = 1) = 1 - \PP(X_n = 0) = 1/\log n$.
We have that almost surely 
$$\lim_{X\to\infty} \frac{\sum_{n\le X} X_n}{\sum_{n\le X} \charf{n\text{ prime}}} = 1.$$
In particular, Cram\'er's model matches the Prime Number Theorem.
It is interesting to test Cram\'er's heuristic as a method for modeling finer aspects of the
distribution of the primes. 
Assuming the Riemann hypothesis, it is known that  
Cram\'er's model matches the truth for primes in intervals of the shape $[X, X + X^{1/2 + \eps}]$.
For $\lambda > 2$, Cram\'er's model suggests that 
\begin{equation*}
\sum_{X < n\le X + (\log X)^\lambda} \charf{n \text{ is prime}} = (1 + o(1))\frac{(\log X)^\lambda}{\log X}
\end{equation*}
since 
$$\sum_{X < n\le X + (\log X)^\lambda} X_n 
= (1 + o(1))\frac{(\log X)^\lambda}{\log X}$$
almost surely.
It turns out that the naive conjecture described above is false, and
\begin{thm}[Maier, \cite{Maier}]\label{thm: Maier original result}Fix $\lambda>1$ and set $H=(\log X)^{\lambda}$. Then, the following holds:
$$\limsup_{X\rightarrow\infty} \frac{\pi(X+H)-\pi(X)}{H/\log X} >1,$$ and $$\liminf_{X\rightarrow \infty} \frac{\pi(X+H)-\pi(X)}{H/\log X} <1.$$
\end{thm}

 Following Maier's remarkable result, similar results on the inequidistribution of other 
 arithmetic sequences in short intervals were obtained by Balog-Wooley \cite{BalogWooley} on numbers representable as sums of two squares, and Tongsomporn and
Steuding \cite{Beatty} on primes in Beatty sequences. Later, Maynard \cite{Maynard} combined Maier's methods with sieve theory to push these results further for prime numbers and numbers representable as the sum of two squares. Additionally, Thorne  \cite{Thorne} has shown the
function field analogue of Theorem \ref{thm: Maier original result}. 

Granville and Soundararajan \cite{GranvilleSound} generalized Maier's work and showed that a 
more general class of arithmetic sequences will be poorly distributed in either short intervals
or short arithmetic progressions. In particular, they demonstrated that this phenomenon is not
specific to the primes. 

Friedlander and Granville \cite{FriedlanderGranville} use Maier's method show the analogue of
Maier's result for short arithmetic progressions. 
Specifically, writing $\pi(X,q,a) = \#\{p\leq X: p\equiv a \mod q, p\text{ prime}\}$, it is
shown that for fixed $A > 2$, large $X$ there exists $q_\pm \in [X,2X]$ such that
$$\pi(X(\log X)^A, q_+, 1) > (1+\delta_A) \frac{X}{\varphi(q_+)(\log X)^{A-1}},$$ 
$$\pi(X(\log X)^A, q_-, 1)< (1-\delta_A) \frac{X}{\varphi(q_-)(\log X)^{A-1}}.$$
There has been further work on these lines by Friedlander, Granville, Hildebrand, and Maier
\cite{FGHM}.


Before we state our general result, we shall discuss some special cases of our main theorem. 
In \cite{Vinogradov}, Vinogradov shows for odd $N$ that  
\[
\#\{p_1,p_2,p_3\in [0,N]: p_1+p_2+p_3 = N\} = \frac{\mf S(N)}{2}\frac{N^2}{(\log N)^3},
\]
where 
\[
    \mf S(N) = \prod_{p | N}\bigg(1 - \frac{1}{(p - 1)^2}\bigg)\prod_{p\nmid N} 
    \bigg(1 + \frac{1}{(p - 1)^3}\bigg).
\]
One might desire to understand the distribution of pairs $(n_1, n_2)$ for which 
$n_1, n_2, N - n_1 - n_2$ are all prime. By the above result, an analogue of 
Cram\'er's model might be to model $\charf{n_1, n_2, N - n_1 - n_2\text{ prime}}$ by 
a random variable $X_{n_1,n_2}$ with 
\[
    \PP(X_{n_1,n_2} = 1) = 1 - \PP(X_{n_1, n_2} = 0) =  \frac{\mf S(N)}{2(\log N)^3}
\]
Unconditionally, Zhan \cite{Zhan} has an asymptotic the count of $(n_1, n_2)$ such that $n_1,  n_2, $ and $N - n_1 - n_2$ 
are prime, and that the $n_i$ are restricted to appropriate short intervals of 
length at least $N^{\frac{5}{8}}(\log N)^c$ for some $c > 0$.
This analogue of Cram\'er's model would suggest that one would be able to count such
$n_i$ in short intervals as short as $(\log X)^\lambda$ for $\lambda > 2$.

A special case of our main theorem provides an answer in the negative to the this 
Cram\'er-type model provides. Specifically, we show:
\begin{samepage}
\begin{corollary}
    Let $N$ be a sufficiently large odd number and $\lambda>\frac{3}{2}$. 
    Let $N\ll X\le N/6, H = (\log X)^\lambda$. 
    There exist $\delta_\lambda^+, \delta_\lambda^- > 0$ such that for $X$ sufficiently large, there exist $x^\pm, y^\pm \in [X,2X]$ such that 
    $$\#\{p_1 \in [x^+, x^+ +H], p_2 \in [y^+, y^++H]: N-p_1-p_2 \textrm{ prime}\}\geq (1+\delta_\lambda^+) \mf S(N)\frac{H^2}{(\log X)^3},$$
    $$\#\{p_1 \in [x^-, x^- +H], p_2 \in [y^-, y^-+H]: N-p_1-p_2 \textrm{ prime}\}\leq (1-\delta_\lambda^-) \mf S(N)\frac{H^2}{(\log X)^3}.$$
\end{corollary}
\end{samepage}

This result can be shown using a combination of Maier's matrix method and the circle method. 
Our main theorem, however, can be applied to a broader class of equations that are beyond the reach of the circle
method. For instance, we can instead consider primes in four term arithmetic progressions, and show 
similar irregularities in the distribution of $(x, y)$ such that $x, x + y, x + 2y, x + 3y$ are all prime.

Next, we introduce the general setting. Let $\Psi = (\psi_1,\hdots, \psi_t):\Z^d\rightarrow \R$ be a system of linear forms.  
\begin{defn}\label{def: complexity}
The \textit{complexity} of a system of forms $\Psi = (\psi_1,\hdots,\psi_t)$ is the smallest integer $s\geq 0$ such that $\{\tilde{\psi_j}: [t]\backslash \{i\}\}$ can be partitioned into $s+1$ classes such that $\tilde{\psi_i}$ is not contained in the affine linear span of any of the classes, where $\Tilde{\psi}$ denotes the homogeneous part of $\psi$. 
\end{defn}
We will restrict ourselves to systems of finite complexity, which measures the linear independence of the homogeneous parts of the forms.
In \cite{GT2}, Green and Tao show that a system has finite complexity if and only if the $\tilde{\psi_i}$ are pairwise pairwise linearly independent. In particular, twin primes and Goldbach's conjecture both have infinite complexity; as a result, we are not able to say anything about such primes using this technology.

In their seminal work, Green and Tao proved the Hardy-Littlewood conjecture for linear patterns with finite complexity. Let us take $\Lambda(n) = \Lambda(|n|)$ for $n\in\ZZ\setminus\set{0}$ and $\Lambda(0) = 0$.

\begin{thm}[Green and Tao, \cite{GT2}]\label{thm: green tao}
Let $\Psi = (\psi_1,...,\psi_t)$ be a system of finite complexity. Then the following asymptotic holds: 
$$ \sum_{\bx \in [-X,X]^d} \prod_{i=1}^t \Lambda(\psi_i(\bx)) = (2X)^d\prod_{p} \beta_p + o_{\Psi}(X^d),$$
where $$\beta_p = p^{-d} \sum_{\bx\in \mathbb{F}_p^d} \prod_{i=1}^t \left(\frac{p}{p-1} \cdot \charf{(\psi_i(\bx),p)=1}\right).$$
\end{thm}

We note that this asymptotic is not exactly what Cram\'er's model predicts, but is properly corrected by the local factors $\beta_p$ for the system $\Psi$.
A random model for the primes along the lines of Cram\'er, correcting for these 
local factors was found by Granville \cite{Granville}. 
In \cite{MSTTHOFU}, Matom\"aki, Shao, Tao and Ter\"av\"ainen, have confirmed that the asymptotic holds in short boxes of length $H\geq X^{5/8+\eps}$: 
\begin{thm}[Matom\"aki, Shao, Tao, and Ter\"av\"ainen, \cite{MSTTHOFU}]\label{thm: MSTT}
Let $X^{5/8+\eps}\leq H \leq X^{1-\eps}$ for some $\eps>0$. Let $\Psi = (\psi_1,...,\psi_t):\Z^d\rightarrow\Z$ have finite complexity. Then we have that $$\sum_{\bx\in [X,X+H]^d}\prod_{i=1}^t \Lambda(\psi_i(\bx)) = H^d \prod_{p} \beta_p + o_{\Psi}(H^d).$$
\end{thm}

We study the limits to equidistribution with the following result.  
\begin{thm}\label{thm: main result}
Let $\lambda>1$, $X$ sufficiently large, and $H=(\log X)^{\lambda}$. Then, there exist $\delta_\lambda^+, \delta_\lambda^- > 0$ such that there exists $\bx^\pm \in [X,2X]^d$ such that $$\sum_{\bx\in \prod_{i=1}^d [x^+_i,x^+_i + H]} \prod_{i=1}^t\charf{\psi_i(\bx) \textrm{ is prime}}\ge(1 + \delta_\lambda^+)\frac{H^d}{(\log X)^t}\prod_{p} \beta_p,$$
and 
$$\sum_{\bx\in \prod_{i=1}^d [x^-_i,x^-_i + H]} \prod_{i=1}^t\charf{\psi_i(\bx) \textrm{ is prime}}\le(1 - \delta_\lambda^-)\frac{H^d}{(\log X)^t}\prod_{p} \beta_p.$$
\end{thm}

\begin{remark}
We note that we expect primes in these patterns to be equidistributed in large enough boxes because it is a linear system. If instead we look at quadratic polynomials, then there are large cone regions where there are no integer solutions. 
For example, there are no primes of the form $m^2 + n^2$ for $0 < m, n\le X$ with 
$0 < \arg(m + ni) < \frac{1}{X}$, for there are no integral points in the region. 
However, equidistribution would predict in most sectors of angle 
$\asymp\frac{1}{X}$, there are $\asymp\frac{X}{\log X}$ primes. 

These irregularities 
are due to the structure of the lattice points for the quadratic in certain sectors. In our 
linear case, there are no such obstructions to regularity; the number of lattice points in 
$\prod_i [x_i, x_i + H]$ remains $(1 + o(1))H^d$ for all $x_i$. Thus, the relative abundance or 
sparsity of linear patterns of primes is therefore due to the structure of the primes and 
not immediately due to the fact that they sit inside the integers. 
\end{remark}

\begin{remark}
We expect to be able to show the analogue of our result in short arithmetic progressions
rather than short intervals \`a la Friedlander and Granville \cite{FriedlanderGranville}. 
\end{remark}

\begin{remark}
    For a fixed $\lambda$ and a range $[X,2X]^d$, we note that if worked out, we can get $\delta_{\lambda}^{\pm}$ to grow exponentially with $t$. 
    In particular, as $t$ grows, we get counts in small boxes that 
    deviate by an arbitrarily large constant factor from the expected count. We 
    are thus able to find deviations that are in a sense more severe than those of 
    Maier.
\end{remark}

We end this section with a brief outline of the rest of the paper. 
In \S2, we will set up the main results from which 
our main theorem follows; in \S2.3 we complete the proof of Theorem \ref{thm: main result} assuming
these two results, to whose proofs the rest of the paper will be devoted. In \S3, we set up the
machinery surrounding nilsequences we shall require in our proof . What remains is then reduced 
to certain estimates sums of nilsequences along primes in arithmetic progressions with moderately 
large moduli, which we show in \S4; it is here where new ideas are brought in.   


\section{Setup} \label{sec:setup}
\subsection{Notation}
In this section, we will specify notation that will be used for the remainder of the paper. We fix a
system of form $\Psi = (\psi_1,...,\psi_t):\Z^d \rightarrow \R$ of finite complexity $s$. 
We write $[N] = \{1,2,...,N\}.$ 
Associated to these forms are the local densities $\beta_p$, as defined in Theorem \ref{thm: green tao}.

Throughout, $p$ always denotes a prime.
Write $P(z) = \prod_{p\leq z} p$. We will use $\mc P$ to denote the set of prime numbers.

We let $\omega:\R_{\geq 1}\to\R_{\geq 1}$ be the Buchstab function, i.e. the unique continuous
function such that
\begin{align*} 
\omega(u) &= \frac{1}{u},\hspace{7pt}  1\leq u\leq 2\\ 
\frac{d}{du}(u\omega(u)) &= \omega(u - 1),\hspace{7pt}  u\ge 2
\end{align*} 
Throughout, $\gamma$ shall denote the Euler-Mascheroni constant.

\subsection{Existence of dutiful moduli} \label{subsec:good_moduli}
\begin{defn}\label{def: dutiful}
A modulus $Q$ is \textit{dutiful} if for all $q\le(\log Q)^{\frac{1}{10}}$, all of the
Dirichlet characters $\chi\mod{(qQ)^{\floor{20s\log\log Q}}}$ satisfy that for any $\beta\in [1/2,1]$ such that 
$L(\beta, \chi) = 0$, the following holds:
\[
    \beta\le 1 - \frac{1}{(\log Q)^2}.
\]

\end{defn}
\begin{remark}
    We note that the choice of $1/10$ here is somewhat arbitrary. It is chosen for notational simplicity and to simply have fewer variables to keep track of through the proof.
\end{remark}

\begin{proposition}\label{prop:good_moduli_exist}
There arbitrarily large $z$ such that $P(z)$ is a dutiful modulus.
\end{proposition}
\begin{proof}
We proceed similarly to the proof of Lemma 1 in \cite{Maier2}.
Write
\[
    M(z) = (P(z)(\ceil*{\log^{1/10} P(z)}!))^{\floor{20s\log\log P(z)}}.
\]
 We record
$\log M(z) = (1 + o(1))20s z\log z$. Take some $z_1$ sufficiently large.
If $P(z_1)$ is dutiful, we are done. Otherwise 
(since $z$ is sufficiently large and $\log M(z) = o(\log^2 P(z))$), there exists
some $\tilde\chi \mod M(z_1)$ with $L(\tilde\beta, \tilde\chi) = 0$
and 
\[
    1 - \tilde\beta\le\frac{c_1}{\log M(z_1)}
\]
for some sufficiently small $c_1$.
Since $z$ is sufficiently large and 
$\log M(z) = (1 + o(1))20s z\log z$, there exists 
$z\ge z_1$ such that 
\[
     \frac{c_1/2}{\log M(z)} < 1 - \tilde\beta\le \frac{c_1}{\log M(z)}
\]
Therefore, $\tilde\chi$ induced $\mod M(z)$ is still an exceptional character 
to modulus $M(z)$.

It follows by Theorem 5.28 in \cite{IK} (Landau/Page's theorem) that there are no other
exceptional characters $\mod M(z)$
So, for all $\chi\mod M(z)$, we have that 
if 
\[
    1 - \beta\le\frac{c_1/2}{\log M(z)},
\]
then $L(\beta,\chi)\neq 0.$
The desired result then follows upon noting that for large $z$,
\[
    \frac{c_1/2}{\log M(z)}\ge\frac{1}{(\log P(z))^2}.
\]
\end{proof}

\subsection{Main estimates}
To prove Therorem \ref{thm: main result}, we shall rely on the following two propositions, from
which we shall prove the Theorem \ref{thm: main result} in the next subsection.

\begin{defn}
For a given modulus $Q$, $\ba\in (\ZZ/Q\ZZ)^d$ is called \textit{admissible} if 

\noindent $(\psi_i(\ba), Q) = 1$ for all $i\le t$.
\end{defn}

\begin{proposition}\label{prop:smooth_ap}
Suppose that $Q = P(z)$ is dutiful, and fix $\eps > 0$. 
Then, for $X\ge \exp((\log Q)^{50})$, admissible $\ba\in (\ZZ/Q\ZZ)^{d}$, we have
\[
    \sum_{\substack{{\bf n}\in [X,2X]^d\\ {\bf n\equiv\bf a}(Q)}}\prod_{i\le t}\charf{\mc P}(\psi_i(n)) 
     = \frac{1}{\#\mc A}\frac{X^d}{(\log X)^t}\prod_{p}\beta_p (1+O(z^{-1})) + o\pfrc{X^d}{(\log X)^t}.
\]
Here, $\mc A\subset (\ZZ/Q\ZZ)^d$ is the set of admissible tuples $\mod Q$.
\end{proposition}
We remark that for almost all $Q\le X^{\frac{1}{3} - \eps}$, 
Ter\"av\"ainen and Shao \cite{TeravainenShao} have shown Proposition 
\ref{prop:smooth_ap}.

We also introduce an estimate for linear patterns of rough numbers. This result follows from 
general work of Matthiessen \cite[Theorem 2.1]{MatthiesenLinear}, 
though in our specific case of rough numbers, we expect 
that it should also follow with a little work from results on linear patterns of primes.
\begin{proposition}\label{prop:rough_patterns}
Fix $u> 1$, take $z$ large, and let $N = z^u$. Then, we have
\[
    \sum_{\substack{{\bf n}\in [N, 2N]^d}}\prod_{i\le t}\charf{(\psi_i(n), P(z))=1} 
    = (1 + o(1))\frac{\#\mc A}{P(z)^d}\left(e^\gamma\omega(u)\right)^t N^d,
\]
where $\mc A$ is the set of admissible tuples $\mod P(z)$.
\end{proposition}
\begin{proof}
    We note that $\charf{(.,P(z))=1}$ is a multiplicative function. Let us take on the notation of \cite{MatthiesenLinear}; we note that $\charf{(.,P(z))=1}$ satisfies the all of the conditions to apply Theorem 2.1 except condition (iii). However, this issue is only an obstacle in the proof that the constructed measure is pseudorandom. However, since $\charf{(.,P(z))}\leq \charf{\text{prime}}$, we can use the pseudorandom measure defined by Green and Tao in \cite{GT2} and proceed with the proof of Theorem 2.1 in \cite{MatthiesenLinear}. 
    
    Take $\tilde{W}(x) = \prod_{p\leq \log\log N} p$. 
    Then, Theorem 2.1 in \cite{MatthiesenLinear} gives us that
    \begin{multline*}\sum_{\substack{{\bf n}\in [N, 2N]^d}}\prod_{i\le t}\charf{(\psi_i(n), P(z))=1} \\= (1+o(1)) \frac{\omega(u)^t}{(\log z)^t} \sum_{(A_i, \tilde{W})=1}\frac{\tilde{W}^t}{\tilde{W}^d\varphi(\tilde{W})^t} \sum_{v\in \Z/\tilde{W}\Z^d}\prod_{i=1}^t \charf{\psi_i(v)\equiv A_i \mod \tilde{W}}
    \end{multline*}
    
    Now let us take $B>0$ large; for $1\leq w_i\leq (\log N)^B$, if both $\charf{(w_i,P(z))=1}\neq 0$ and $p\mid w_i\implies p\mid \tilde{W}$, then $w_i=1$. So the sum over $w_i$'s becomes 
    \begin{multline*}(1+o(1)) \frac{\omega(u)^t}{(\log z)^t} \sum_{(A_i, \tilde{W})=1}\frac{\tilde{W}^t}{\tilde{W}^d\varphi(\tilde{W})^t} \sum_{v\in \Z/\tilde{W}\Z^d}\prod_{i=1}^t \charf{\psi_i(v)\equiv A_i \mod \tilde{W}} \\ = (1+o(1)) \frac{\omega(u)^t}{(\log z)^t} \cdot \prod_{p\leq\log\log N}\beta_p.\end{multline*}
    Now, we can see that $$\prod_{p\leq \log\log(N)} \beta_p = \frac{\#\calA}{P(z)^d}\cdot \frac{P(z)^t}{\varphi(P(z))^t} \cdot  (1 + O(z^{-1} + (\log\log N)^{-1}),$$
    as both sides can be related to $\prod_{p}\beta_p$. Finally, $$\frac{P(z)}{\varphi(P(z))} = \prod_{p\leq z} (1-1/p)^{-1} \sim e^{\gamma}(\log z),$$
    as desired. 
\end{proof}

\subsection{Proof of main theorem}\label{subsec:main_thm_pf}

We now prove the main theorem using Maier's matrix method, assuming Propositions 
\ref{prop:smooth_ap} and \ref{prop:rough_patterns}. 

Take $Q = P(z)$ a sufficiently large dutiful modulus, and let 
\begin{equation}\label{eq:X_from_Q}
    X = Q\floor*{\exp((\log Q)^{100})/Q}
\end{equation}
so that $Q | X$. 
Next, take $U = z^u$, for some $u$ to be determined later. 
As in the Maier matrix method, we shall estimate the following sum
\begin{equation}\label{eq:maier_double_sum}
    \Sigma = \sum_{\ba\in [U, 2U]^d}
    \sum_{\substack{\bn\in [X, 2X]^d\\\bn\equiv\ba(Q)}}
    \prod_{i\le t}\charf{\mc P}(\psi_i(\bf n)).
\end{equation}
in two different ways. 

On one hand, we note that the inner sum only contributes if $\ba$ is admissible. Then, we 
may estimate the inner sum with Proposition \ref{prop:smooth_ap}, and then execute the second
sum in the second line with Proposition \ref{prop:rough_patterns}. By this, we obtain 
\begin{align*}
    \Sigma &= \sum_{\substack{\ba\in [U, 2U]^d\\ \ba \text{ admissible}}}
    \sum_{\substack{\bn\in [X, 2X]^d\\\bn\equiv\ba(Q)}}
    \prod_{i\le t}\charf{\mc P}(\psi_i(\bn))\\
    &= 
    (1 + o(1))\frac{X^d}{(\log X)^t} \cdot \prod_p\beta_p\cdot \frac{1}{\#\mc A}\sum_{\ba\in [U, 2U]^d}
    \charf{(\psi_i(\ba), Q) = 1\forall i}\\ 
    &= (1 + o(1))\frac{U^d X^d}{Q^d(\log X)^{t}}\cdot (e^\gamma\omega(u))^t\cdot \prod_{p}\beta_p.
\end{align*}

On the other hand, rearranging, we have
\[
    \Sigma =
    \sum_{\br\in [(X - U)/Q, 2(X - U)/Q]^d}\sum_{\bn\in R_{\br}}\prod_{i\le t}\charf{\mc P}(\psi_i(\bn)),
\]
where 
\[
    R_{\br} = \prod_{j\le d} [Q\br_j + U, Q\br_j  + 2U].
\]
Since $\Sigma$ is the average over all $\br\in [(X - U)/Q, 2(X - U)/Q]^d$, there must exist $\br^\pm\in [(X - U)/Q, 2(X - U)/Q]^d$ such that
\[
    \sum_{\bn\in R_{\br^+}}\prod_{i\le t}\charf{\mc P}(\psi_i(\bn))\ge 
    (1 + o(1))(e^\gamma\omega(u))^{t} U^d (\log X)^{-t}\prod_p\beta_p,
\]
\[
    \sum_{\bn\in R_{\br^-}}\prod_{i\le t}\charf{\mc P}(\psi_i(\bn))\le 
    (1 + o(1))(e^\gamma\omega(u))^t U^d (\log X)^{-t}\prod_p\beta_p.
\]
The desired result follows from the fact that for any $u$, there exist 
$u^+, u^-\ge u$ such that $e^\gamma\omega(u^-) < 1 < e^\gamma\omega(u^+)$ (see Lemma 4 in \cite{Maier}, for 
example). Finally, let us take $u = 101\lambda$ so that $z^u \geq (\log X)^\lambda$. Let us look at the
short boxes of length $U^+ = z^{u^+}$ (resp. $U^- = z^{u^-}$); applying the pigeonhole within the box,
we can see that there is a box of length $(\log X)^\lambda$ with an exceptionally large (resp. small)
number of primes.

\begin{remark}
    We remark that it is not necessary to take $Q = P(z)$ in these arguments; more generally we could
    take $Q = \prod_{p\in \mc S} p$ for $\mc S$ a subset of the primes that grows with $z$. 
\end{remark}

\section{The Green-Tao Machinery}\label{sec:background}

Let $f:\Z\rightarrow \R$ be an arithmetic function, and let $\Psi = (\psi_1,\hdots,\psi_t):\Z^d \rightarrow \R^t$ be a system of linear forms with complexity $s$. Assume that $$\EE f(n) = \delta + o(1)$$ for some constant $\delta>0$. Consider the following sum: 
\begin{equation}\label{eq: green tao original sum}
    \sum_{\br \in \mc R\cap \Z^d} \prod_{i=1}^t f(\psi_i(\br)).
\end{equation}

The method of Green and Tao in \cite{GT2} allows us to find asymptotics for the above sum, assuming ``nice'' properties of $f$. We later will take $f$ so that we may count either primes or rough numbers. In either scenario, we must assume that $f$ is bounded by a ``constant-like'' function --- a pseudorandom measure. 

\begin{defn}
A \textit{$k$-pseudorandom measure} $\nu:[N]\rightarrow\R$ is a function that satisfies:
\begin{itemize}
    \item $\nu$ is a measure: 
    $$\EE \nu(n) = 1+o(1).$$
    \item $\nu$ satisfies the linear forms condition: Let $m\leq k2^{k-1}$ and $t\leq 3k-4$. Let $(L_{ij})_{1\leq i\leq m,1\leq j\leq t}$ be a matrix of rational numbers of height $\ll k$ where no row is zero or a rational multiple of another. Let $\psi_i:(\Z/N\Z)^t\rightarrow\Z/N\Z$ be linear forms given by $\psi_i(x) = \sum_{j=1}^t L_{ij} x_j + b_i$. Then if we consider the natural extension of $\nu$ to $\Z/N\Z$, we have that $$\EE_{x\in (\Z/N\Z)^t}(\nu(\psi_1(x))\hdots \nu(\psi_m(x))) = 1 + o(1).$$
\end{itemize}
\end{defn}
\begin{remark}
    We would like to note that originally Green and Tao included a correlation condition, but this can be dropped due to the work of Conlon, Fox, and Zhao \cite{ConlonFoxZhao}, and Dodos and Kanellopoulos \cite{DodosKanellopoulos}.
\end{remark}

For such $f$ bounded by pseudorandom measures, we can relate (\ref{eq: green tao original sum}) to the Gowers norm of $f$. In particular, after a technical reduction of rewriting $\Psi$ in $s-$normal form (which is always possible as described in section 4 of \cite{GT2}) we can apply the generalized von Neumann theorem:

\begin{thm}[Green and Tao, Proposition 7.1 \cite{GT2}]\label{thm: von-Neumann}
Let $\nu$ be a $k(s)$-pseudorandom measure and $f_1,\hdots, f_t :\Z\rightarrow \R$ are functions $|f_i(x)|\leq \nu(x)$ . If $\Psi$ is a system in $s$-normal form. If $$\min \|f_j\|_{U^{s+1}[N]} \leq \delta,$$ then we have that for some $c>0$, $$\sum_{n\in [N,2N]^d}\prod_{i\in [t]} f_i(\psi_i(n))\ll_{\Psi, k(s)} \kappa(\delta)N^d + o(N^d),$$
where $\kappa(\delta)$ is a function that approaches $0$ as $\delta\rightarrow 0.$
\end{thm}
\begin{proof}[Proof sketch]
    We proceed as in the proof of Proposition 7.1 in \cite{GT2}; first we take $\eps = \delta^c$ in Corollary A.3 (for $c$ yet to be determined) and transfer the problem to $\Z/N'\Z$, for $N'$ a sufficiently large constant multiple of $N$. In particular, for $F$ a Lipschitz function of Lipschitz constant $O(\delta^{-c})$ and bounded by $1$, we want to show that $$\EE_{(\Z/N'\Z)^d}F(n) \prod_{i\in [t]} f_i(\psi_i(n)) = O(\kappa(\delta)).$$
    Next, we expand $F$ under a Fourier expansion: for any $X>1, J = X^d$ a 
    Fourier expansion
    $$F(n) = \sum_{j=1}^J c_j e(m_j n/N') + O(\delta^{-c}(\log X)/X)$$
    for some coefficients $c_j\ll 1$.
    We get that 
    \begin{align*}\EE_{n\in (\Z/N'\Z)^d} &F(n) \prod_{i\in [t]} f_i(\psi_i(n))  \\
    &= \sum_{j=1}^J c_j \EE_{n\in (\Z/N'\Z)^d} e(m_jn/N') \prod_{i\in [t]} f_i(\psi_i(n)) + O_{k(s)}\left(\frac{\delta^{-c} (\log X)}{X}\right).\end{align*}
    Here we have applied the linear forms condition for $\nu$ to bound the error term from the Fourier expansion. Let us take $X$ growing slowly in $1/\delta$. 

    It is observed by Green and Tao, that the individual exponential sum correlation bounds will be implied by showing $$\EE_{(\Z/N'\Z)^d} \prod_{i\in [t]} f_i(\psi_i(n)) = O(X^{-d}\kappa(\delta)).$$ This then reduces us to Proposition 7.1'' in Appendix A of \cite{GT2}, which states that $$\EE_{(\Z/N'\Z)^d} \prod_{i\in [t]} f_i(\psi_i(n)) \leq \|f_i\|_{U^{s+1}(\Z/N'\Z)} + o(1).$$
    Now applying our assumption and setting $X = (1/\delta)^{1/2d}$, $c=1/4d,$ and taking $\kappa(\delta) = \delta^{1/6d}$, we obtain the above bound. 
\end{proof}
\begin{remark}
    We note that we will eventually take $\delta(N)$ as function satisfying that $\delta(N)\rightarrow 0$ and $N\rightarrow \infty,$ so this upper bound will become $o(N^d).$
\end{remark}

In order to achieve our desired estimate on primes or rough numbers, we will want to use the following quantitative inverse theorem for the Gowers norm; a statement of this form was first achieved by Manners \cite{Manners}, but we will use the version obtained by Tao and Ter\"aiv\"ainen in \cite{TT21}
\begin{thm}[\cite{TT21}, Theorem 8.3]\label{thm: manners inverse theorem}
Let $\nu$ be a pseudorandom measure such that $$\|\nu-1\|_{U^{2k}[N]} \leq \delta^{C_0}$$ for some constant $C_0$ sufficiently large in $k$. Let $f$ be a $\nu$-bounded function such that $$\|f\|_{U^k[N]}\geq \delta.$$
Then there exists a $(k-1)$-step nilmanifold $G/\Gamma$ of dimension $O(\delta^{-O(1)})$ and complexity at most $\exp\exp O(\delta^{-O(1)})$, a Lipschitz function $F:G/\Gamma\rightarrow \C$ of complexity at most $\exp\exp O(\delta^{-O(1)})$, and a polynomial map $g:\Z\rightarrow G$ such that $$\EE f(n) F(g(n) \Gamma) \gg \exp(-\exp(O(\delta^{-O(1)})).$$

\end{thm}

Thus, it will remain to bound sums of the form $\EE f(n) F(g(n) \Gamma)$. Our approach is motivated by the splitting of sums into ``major'' and ``minor'' arcs in the circle method; for 1-step nilsequences, this is the same process. We take $G/\Gamma$ to be $s$-step nilmanifold with Mal'cev basis $\calX$. Before introducing the factor theorem, let us first let us recall some definitions: 
\begin{defn}\label{def: rational sequence}
$\gamma\in G$ is \textit{$Q$-rational} if $\gamma^r\in \Gamma$ for some $0\leq r\leq Q$. 
\end{defn}

\begin{defn}\label{def: smooth sequence}
$f(n):\Z\rightarrow G$ is a \textit{(M,N)-smooth} if $d(f(n),\text{id}_G) \leq M$ and 

\noindent$d(f(n),f(n-1))\leq M/N$ for all $n\in \Z$. Here we are using the metric induced on $G$ by $\calX$. 
\end{defn}

\begin{defn}\label{def: totally equidistributed}
    We say that a sequence $(g(n)\Gamma)$ is \textit{totally-$\delta$-equidistributed} if $$\bigg|\EE_{P}F(g(n)\Gamma) - \int_{G/\Gamma} F\bigg|\leq \delta \|F\|_{\text{Lip}},$$
    for all Lipschitz functions $F:G/\Gamma\rightarrow\C$ and $P\subset [N]$ an arithmetic progression of length at least $\delta N.$
\end{defn}

\begin{thm}[{\cite[Theorem 1.19]{GT3}}]\label{thm: factor theorem}
Let $G/\Gamma$ be a $m$-dimensional manifold with a $s$-step filtration and $N\geq 0$, $0\leq \delta\leq 1/2$. Let $\calX$ be a $\delta_0$-rational Mal'cev basis adapted to the filtration and $g\in\mathrm{poly}(\Z,G)$ of degree $d$. Then there exists a number $\delta_0\geq \delta\geq \delta_0^{O(1)}$, a rational subgroup $G'\subset G$ with Mal'cev basis $\calX'$ for $G'/\Gamma'$, which are $M$-rational combinations of the basis elements of $\calX$, and a decomposition $$g(n) = \xi(n)g'(n)\gamma(n)$$ such that $\xi,g',\gamma$ are polynomial sequences and 
Let $G/\Gamma$ be a $m$-dimensional manifold with a $s$-step filtration and $N\geq 0$, $0\leq \delta\leq 1/2$. Let $\calX$ be a $\delta_0$-rational Mal'cev basis adapted to the filtration and $g\in\mathrm{poly}(\Z,G)$ of degree $d$. Then there exists a number $\delta_0\geq \delta\geq \delta_0^{O(1)}$, a rational subgroup $G'\subset G$ with Mal'cev basis $\calX'$ for $G'/\Gamma'$, which are $M$-rational combinations of the basis elements of $\calX$, and a decomposition $$g(n) = \xi(n)g'(n)\gamma(n)$$ such that $\xi,g',\gamma$ are polynomial sequences and 
\begin{itemize}
    \item $\xi(n):\Z\rightarrow G$ is $(\delta^{-1},N)$-smooth,
    \item $g'(n):\Z\rightarrow G'$ is totally-$\delta^{O(1)}$-equidistributed in $G'/\Gamma'$ using the metric induced by $\calX'$,
    \item $\gamma(n):\Z\rightarrow G$ is $\delta^{-1}$-rational and $(\gamma(n)\Gamma)$ is periodic with period at most $\delta^{-1}$. 
\end{itemize}
\end{thm}
Using this factor theorem in \cite{GT4}, Green and Tao showed that the Mobius function is orthogonal to nilsequences; in particular any nilsequence with nontrivial equidistributed part belongs to the ``minor arc'' and the periodic parts can be bounded using Siegel's theorem due to the size of the period being small. In our application of Proposition \ref{prop:smooth_ap}, however, we need to modify the factor theorem slightly to deal with ``major arcs'' with larger moduli. We remark that a similar theorem is proved by Matthiesen in \cite{MatthiesenPosDef} (Theoren 6.4).

\subsection{A $Q$-factor theorem}\label{subsec:factor_theorem}

\begin{thm}\label{thm: Q-factor theorem}
Let $G/\Gamma$ be a $m$-dimensional manifold with a $s$-step filtration, $N,Q\geq 0$, $0\leq \delta_0\leq 1/2$, and $0\leq a<Q<N^{1/2}$. Let $\calX$ be a $\delta_0^{-1}$-rational Mal'cev basis adapted to the filtration and $g\in \text{poly}(\Z,G)$ of degree $d$. Then there exists a number $\delta_0\geq \delta \geq \delta_0^{O(1)}$, a rational subgroup $G'\subset G$ with Mal'cev basis $\calX'$ for $G'/\Gamma'$, which are $M$-rational combinations of the basis elements of $\calX$, and a decomposition $$g(n) = \xi(n)g'(n)\gamma(n)$$ such that $\xi,g',\gamma$ are polynomial sequences and 
\begin{itemize}
    \item $\xi(n):\Z\rightarrow G$ is $(\delta^{-1},N)$-smooth,
    \item $g'(Qn+a):\Z\rightarrow G'$ is totally-$\delta^{O(1)}$-equidistributed in $G'/\Gamma'$ using the metric induced by $\calX'$,
    \item $\gamma:\Z\rightarrow G$ satisfies that $\gamma(Qn+a)$ is $\delta^{-1}$-rational, $(\gamma(n)\Gamma)$ is periodic with period $(Qq)^d$ with $q$ at most $\delta^{-1}$, and $(\gamma(Qn+a)\Gamma)$ is periodic with period at most $\delta^{-O(1)}$. 
\end{itemize}
\end{thm}
We achieve the above theorem using morally the same proof as that of Theorem \ref{thm: factor theorem} in \cite{GT3}, only with an extra dependency on an exterior $Q$. In particular, we start with the quantitative Leibman theorem from \cite{GT3}. 

\begin{thm}[{\cite[Theorem 1.16]{GT3}}]\label{thm: quant Leibman}
Let $G/\Gamma$ be a $m$-dimensional manifold with a $s$-step filtration. Let $N\geq 1$ and $0\leq \delta \leq 1/2$. If $\calX$ is a $\delta^{-1}$-rational Mal'cev basis adapted to the filtration and $g\in \mathrm{poly}(\Z,G)$, then at least one of the following is true:
\begin{itemize}
    \item $(g(n)\Gamma)$ is $\delta$-equidistributed in $G/\Gamma$,
    \item There is a nontrivial horizontal character $\eta: G\rightarrow \R/\Z$ such that $|\eta|\ll \delta^{-O(1)}$ and $\|\eta\circ g(n)\|_{C^\infty[N]} \ll \delta^{-O(1)}$. 
\end{itemize}
\end{thm}

\begin{proof}[Proof of Theorem \ref{thm: Q-factor theorem}]
Assume $g(Qn+a)$ is not totally-$\delta^{O(1)}$-equidistributed on $G/\Gamma$. Then by Theorem \ref{thm: quant Leibman} there is some $q'\ll \delta^{-O(1)}$ and $|\eta|\ll \delta^{-O(1)}$ such that $$\|\eta\circ g(Qq'n+a)\|_{C^\infty[N]} \ll \delta^{-O(1)}.$$
Let $\psi:G\rightarrow \R^m$ be our map to Mal'cev coordinates (recall that the dimension of $G$ is $m$), then we know that $\eta\circ g(n) = k\cdot \psi(g(n))$ for some vector $k\in \mathbb{Q}^m$ with $|k|\ll \delta^{-O(1)}$. Since $g(n)$ is a polynomial map, we have that $$\psi(g(n)) = \sum_{j=1}^d t_j \binom{n}{j},$$ for some $t_j\in \R^m$. Restricting to an arithmetic progression mod $Q$, we have that $$\eta\circ g(Qq'n+a) = k\cdot \psi(g(Qq'n+a)) = \sum_{j=1}^d (k\cdot t_j) \binom{Qq'n+a}{j}.$$

We wish to rewrite the above as $$\sum_{j=1}^d \alpha_j \binom{n}{j},$$ so that we can apply information from the $C^\infty[N]$ norm of $\eta\circ g(n)$. We note that the $C^\infty[N]$ is independent of the shift by $a$, so we can take $a=0$. We know that $\alpha_d = (Qq')^d (k\cdot t_d)/d!$, so we have that $$\|(Qq')^d (k\cdot t_d)/d!\|_{\R/\Z} \ll \frac{\delta^{-O(1)}Q^d}{N^d}.$$
We note that $$\alpha_{d-1} = (k\cdot t_d)(Qq')^{d-1} \cdot \frac{d(d-1)}{2d!}(1-Qq') + (k\cdot t_{d-1})(Qq')^{d-1}\cdot \frac{1}{(d-1)!}.$$
However, since $Q\delta^{-O(1)}/N \ll N^{-1/3}$ we can see the contribution from the terms with $(k\cdot t_d)$ will be inconsequential to $\|\alpha_{d-1}\|_{\R/\Z}$ and that $$\|(Qq')^{d-1}(k\cdot t_{d-1})\|_{\R/\Z} \ll_{d} \frac{\delta^{-O(1)} Q^{d-1}}{N^{d-1}}.$$
Similarly, we achieve that $$\|(Qq')^{j} (k\cdot t_{j})\|_{\R/\Z} \ll_d \frac{\delta^{-O(1)}Q^{j}}{N^j},$$
for each $j=1,...,d.$ 

So, let us pick $u_j\in \R^m$ such that $Q^j (k\cdot u_j) \in \Z$ and $|t_j-u_j|\ll \frac{\delta^{-O(1)}}{N^j}$. We define $\xi(n)$ as the sequence such that $$\psi(\xi(n)) = \sum_{j=1}^d \binom{n}{j}(t_j-u_j).$$ Next, we pick the vectors $v_j$ with denominators of the form $(Qq)^j$ with $q \ll \delta^{-O(1)}$ and such that $k\cdot u_j = k\cdot v_j$. We define $\gamma(n)$ as the sequence such that $$\psi(\gamma(n)) = \sum_{j=1}^d \binom{n}{j}v_j.$$
So, we take $g_1(n) = \xi^{-1} g \gamma^{-1}(n).$ It is clear via the arguments of \cite{GT3} that $\xi(n)$ is $(\delta^{-O(1)},N)$-smooth and $\gamma$ has the correct periodicity properties. So, it remains to construct $g'$ totally $\delta^{-O(1)}$-equidistributed. However, $g_1$ is not necessary such; if $g_1$ is not totally equidistributed then we can repeat the argument above iteratively, forming $\xi_1$ and $\gamma_1$. We note that $\xi(n)\xi_1(n)$ will still be $(\delta^{-O(1)},N)$ smooth and $\gamma_1(n)\gamma(n)$ periodic. Thus, we repeat the procedure until we have $g'(Qn+a)$ totally $\delta^{-O(1)}$-equidistributed. 
\end{proof}

\section{Linear patterns in smooth arithmetic progressions} \label{sec:smooth_ap}
In this section, we will prove Proposition \ref{prop:smooth_ap}. 
We recall our goal is to show that $$\sum_{\substack{\bn\in [X,2X]^d \\ \bn \equiv \ba \mod Q}} \prod_{i=1}^t \Lambda(\psi_i(\bn)) = (1+o(1))\frac{1}{\#\calA} X^d \prod_{p}\beta_p,$$ where $Q$ is a dutiful modulus, $\ba\in \calA \subset (\Z/Q\Z)^d$ an admissible tuple, and $\beta_p$ the local factors associated to $\Psi$. In particular, we wish to upper bound $$\sum_{\substack{\bn\in [X,2X]^d\\ \bn\equiv \ba\mod Q}} \prod_{i=1}^t \Lambda(\psi_i(\bn)) - \frac{Q^d}{\#\calA} \prod_{p}\beta_p.$$
We note that since $Q = P(z)$, 
\begin{align*} 
\#\calA = \#\{\ba \mod Q: (\psi_i(\ba),Q) = 1 \forall i\} 
&= \prod_{p\leq z} \sum_{\bx\in \mathbb{F}_p^d}\prod_{i=1}^t\charf{(\psi_i(\bx),p)=1} \\
&= Q^d\pfrc{\varphi(Q}{Q}^t\prod_{p\leq z}\beta_p.
\end{align*}
So, it remains to upper bound $$\sum_{\substack{\bn\in [X,2X]^d \\ \bn\equiv \ba \mod Q}} \prod_{i=1}^t \left(\Lambda(\psi_i(\bn)) - \frac{Q}{\varphi(Q)}\right).$$

Now, we apply the Green-Tao method described in the previous sections; we note that for our choice of $X$ in (\ref{eq:X_from_Q}), we can see that $Q = P(z)$ for $z\leq (\log X)^{1/50}.$ So, we can use Proposition 8.4 from \cite{TT21}, taking $\eps = 1/50$, to construct a pseudorandom measure on $\Lambda(Qn+b)$ satisfying the desired conditions. 
We see that this reduces to showing that for
$(b,Q)=1$, 
\begin{equation}\label{eq: Gowers norm smooth ap}
\bigg|\sum_{\substack{n\sim N\\ n\equiv b\mod Q}} (\Lambda(n)-Q/\varphi(Q))F(g(n)\Gamma)\bigg| 
= o(N/Q)
\end{equation} 
for $F(g(n)\Gamma)$ a $s$-step nilsequence. We apply Theorem \ref{thm: Q-factor theorem} and write $g(n) = \xi(n)g'(n)\gamma(n)$. First, we will follow Green and Tao in \cite{GT4} for when $g'$ is nontrivial, i.e. our ``minor arc'' case. 

\subsection{Totally equidistributed case}\label{subsec:smooth minor arcs}
Let us assume that $g'$ is nontrivial, in which case we have that $$\sum_{\substack{n\sim N\\ n\equiv b\mod Q}} F(g(n)\Gamma) = o(N/Q)$$ by definition of $g'$ being totally equidistributed. Next, let $r = \delta^{-O(1)}$ be the periodicity of $\gamma(Qn+b)\Gamma$. We can then bound the remaining sum by $$\sum_{c\mod r}^* \bigg|\sum_{\substack{n\sim N\\ n\equiv b \mod Q\\ n\equiv c \mod r}} \Lambda(n) F(\xi(n)g'(n)\gamma_c\Gamma)\bigg|,$$
where the sum over $c$ is restricted to those moduli that can also be $b\mod Q$. 
Additionally, we restrict to progressions of $P$ of difference $r$ and length $\delta^{O(1)}N$ so that $\xi(n)$ will be near constant on $P$. Then we can see that it suffices to show that $$\sum_{\substack{n\in P\\ n\equiv b\mod Q\\ n\equiv c \mod r}} \Lambda(n) F(\xi_Pg'(n)\gamma_c\Gamma)  = o\left(\frac{|P|}{\lcm(r,Q)}\right).$$

We now apply Vaughan's decomposition of $\Lambda$ to reduce the problem to Type I and Type II sums. In particular, if the above sum is $\gg \eps |P|/Qr$, then at least one of the following is true:
\begin{itemize}\item For $\gg \delta^{O(1)}K$ integers in $[K,2K]$ with $K\leq N^{2/3}$, 
\begin{equation}\label{eq: type I}\sum_{\substack{n\in P/k\\ kn\equiv b \mod Q\\ kn\equiv c \mod r}} F(\xi_Pg'(kn)\gamma_c\Gamma)\gg \eps \frac{\delta^{O(1)} |P|}{K\lcm(r,Q)},\end{equation}
\item There is a $K,W$ such that $N^{1/3}\ll K\ll N^{2/3}$ and $N\ll KW\ll N$ such that for $\gg \delta^{O(1)}K^2$ pairs $(k,k') \in [K,2K]^2$, such that \begin{equation}\label{eq: type II}\sum_{\substack{w\in P/K\\ kw,k'w\equiv b\mod Q\\ kw,k'w \equiv c \mod r}} F(\xi_P g'(kw)\gamma_c\Gamma) \overline{F(\xi_P g'(k'w) \gamma_c\Gamma} \gg \eps \delta^{O(1)}\frac{|P|}{K\lcm(r,Q)}.\end{equation}
\end{itemize}

Let us first consider the Type I case, although the arguments for both will be similar. Then we find that after pigeonholing to a specific residue of $k \equiv \tilde{k} \mod Qr$, for $\gg \delta^{O(1)} Q^{-1} K$ integers in $[K,2K]$, $$\sum_{n\in P/k} F(\xi_P g'(k(Qrn+b)\gamma_c\Gamma) \gg \eps \delta^{O(1)} |P|/KQr.$$
Let us write $g_k(n) = g(kn)$ and $g_{k,b}(n)= g(k(n+b))$; this will still be a polynomial sequence. The lower bound (\ref{eq: type I}) implies that $g_{k,b}(n)$ is not equidistributed on progressions of length $\delta^{O(1)}Q$, in particular, by the argument in \cite[\S3]{GT4}, there is a character $\eta_k:G\rightarrow\R/\Z$ with $|\eta_k|\ll \delta^{-O(1)}$ with $$\|\eta_k \circ g_k(Qrn)\|_{C^\infty[N/k]}\ll \delta^{O(1)}.$$
Now, we pigeonhole again to fix a character $\eta$.  
If $$\eta\circ g'(n) = \alpha_d n^d+ ... + \alpha_0,$$
then we know that $$\eta\circ g_k(Qrn) = (Qr)^d \alpha_d n^d + ... + \alpha_0.$$
Thus, we have that for $\gg \delta^{O(1)}Q^{-1}K$ integers $k\equiv \Tilde{k} \mod Qr$, we have that $$\|q_jQ^jr^j\alpha_j k^j\|\ll \left(\frac{\delta^{O(1)}K}{N}\right)^j.$$

By Lemma 3.3 of \cite{GT4}, which uses that for $t=4^j$, most admissible integers can be represented as the sum of $t$ $j$th powers, this implies that there are $\gg \delta^{O(1)} Q^{-1} K^j$ such integers $m$ in $[K^j/Q]$ such that for $k_j \equiv t\tilde{k}^j \mod Qr$, $m$ will satisfy that \begin{equation}\label{eq: restricted minor arc inequality}\|q_jQ^jr^j(Qm+t\tilde{k}^j) \alpha_j\|_{\R/\Z} \ll \delta^{O(1)}(K/N)^j,\end{equation} for some $q_j\ll \delta^{O(1)}$. Then by Lemma 3.4 in \cite{GT4}, we have that there exists a $q_j' \ll \delta^{O(1)}$ such that $$\norm*{q_jq_j'Q^{j+1}\alpha_j}_{\R/\Z} \ll \delta^{O(1)}Q/N^j.$$
Now, if $\alpha_j = \frac{a_j}{q_jq_j'Q^{j+1}}+ \beta_j$, with $\beta_j \leq \delta^{O(1)}Q^{-j}N^{-j}$, then we return to the earlier inequality (\ref{eq: restricted minor arc inequality}) and see that this implies $$\norm{\frac{ma_j}{q_j'} + \frac{t\tilde{k}^ja_j}{q_j'Q} + O(\delta^{O(1)}(K/N)^j)}_{\R/\Z} \ll \delta^{O(1)} (K/N)^j. $$
Since $(b,Q)=1$, we can see that the $\tilde{k}$ chosen above by the pigeonhole principle will satisfy that $\tilde{k}^j \neq 0 \mod Q$. So, since $Q \leq N^{1/3}\leq N/K$, we have that $a_j$ must be divisible by one value of $Q$. Thus, we have that for (\ref{eq: type I}) to hold, we must have $$\alpha_j = \frac{a_j'}{q_jq_j'Q^j} + O(\delta^{O(1)}(QN)^{-j}),$$ for each $j$. Now, we recall that $g'(n)$ was defined such that $g'(Qn+a)$ is totally-$\delta^{O(1)}$-equidistributed. We see that the above relation contradicts this assumption, and hence if $g'$ is nontrivial, (\ref{eq: type I}) can not occur. 

Next, we treat the Type II case, which follows a similar argument. Under the assumption of (\ref{eq: type II}), we get that after pigeonholing in residues of $k,k'$ mod $Qr$, for $\gg \delta^{O(1)}Q^{-2} K^2$ pairs $(k,k')\in [K,2K]^2$, we have that $$\sum_{n\in P/K} F(\xi_P g'(k(Qrn+b)\gamma_c\Gamma)F(\xi_Pg'(k'(Qrn+b))\gamma_c\Gamma) \gg \eps \delta^{O(1)} \frac{|P|}{KQ}.$$
Next, we write $g_k(n) = g(kn)$ and $g_{k,k',b}(n) = (g_k(n+b),g_{k'}(n+b))$, which will be a polynomial sequence on $G/\Gamma\times G/\Gamma.$ Again, we have that this will not be $\delta^{O(1)}$-equidistributed on progressions of length $\delta^{O(1)}$, so we get a character $\eta_{k,k'}:G\times G\rightarrow\R/\Z$ with $$\|\eta_{k,k'}\circ g_{k,k',b}(Qn)\|_{C^\infty[N/K]} \ll \delta^{O(1)}.$$
As a result, we have that $$\|q_j Q^j (k^j + k'^j) \alpha_j\|\ll \delta^{O(1)} (K/N)^j,$$ for $\gg \delta^{O(1)}Q^{-2}K^2$ pairs $(k,k')$ with $k,k'$ having fixed residues mod $Qr$. We can pigeonhole and find a $k'$ such that it occurs in $\gg \delta^{O(1)} Q^{-1}K$ pairs $(k,k')$. This reduces us to equation (\ref{eq: restricted minor arc inequality}) and the rest of the proof follows as above. Hence, if $g'$ is nontrivial, we have that (\ref{eq: Gowers norm smooth ap}) holds. This completes the ``minor arc'' case.

\subsection{Periodic case}\label{subsec:smooth major arcs}
We can see that for the proof of Proposition \ref{prop:smooth_ap}, it remains to show that (\ref{eq: Gowers norm smooth ap}) holds for the case when $g'=1$, i.e. the ``major arcs''. By partial summation (to remove the smooth part $\xi(n)$ from the decomposition in Theorem \ref{thm: Q-factor theorem}), we are reduced to 
showing the following.
\begin{proposition}
Suppose that $Q=P(z)$ is a dutiful modulus, $X$ as in (\ref{eq:X_from_Q}), 
$N\gg X^{1/10}$ and $M\gg\frac{N}{(\log N)^{O(1)}}$. 
Then, for any $q\ll(\log X)^{O(1)}$, $G/\Gamma$ a $s$-step nilmanifold, and any $(qQ)^d$-periodic polynomial sequence
$\gamma : \ZZ\to G$, we have that for some absolute $\delta>0$:
\begin{align*}
    \sum_{\substack{N < n\le N + M\\ n\equiv a(Q)}}
    (\Lambda(n) - Q/\varphi(Q))F(\gamma(n)\Gamma)\ll
    \frac{N}{\varphi(Q)}(\log Q)^{-\delta}.
\end{align*}
\label{prop:maj_arc_main_est}
\end{proposition}

We prove this at the end of the section, after preparing some technical results that
shall be necessary in the proof.

\begin{lemma}\label{lem:periodic_nilseq_char}
We have that for any Dirichlet character $\chi\mod q$ of conductor $q_1$ not divisible by the square of 
an odd prime, and any $q$-periodic $s$-step nilsequence $n\mapsto F(\gamma(n)\Gamma)$, and any $\eps>0,$
\begin{equation*}
    \sum_{x(q)}\chi(x)F(\gamma(x)\Gamma)\ll q^{1 + \eps} q_1^{-c 2^{-s}}.
\end{equation*}
for some absolute $c > 0$.
\end{lemma}
\begin{proof}
    By the forward direction of the inverse theorem for the Gowers norm 
    \footnote{See, for example, the proof of Proposition 1.4 in Appendix G of \cite{GT5}
    for $\Z/q\Z$;
    one can check that the loss is only polynomial during the differencing.}, for any $\eta_s>0$, if we know: 
    \[
        \bigg|\sum_{x(q)}\chi(x)F(\gamma(x)\Gamma)\bigg|\ge q^{1 + \eps}q_1^{-\eta_s},
    \]
    we have that 
    \[
        \norm{\chi}_{U^s[q]}\gg q^\eps q_1^{-O(\eta_s)}.
    \]
    This is false, however, for sufficiently small $\eta_s$, since the following bound holds (which then can be inputted into the definition of the Gowers norm): 
    \begin{multline*}
        \frac{1}{q}\sum_{x(q)}\chi(x + h_1)\dots\chi(x + h_s)\conj\chi(x + h_{s + 1})\dots\conj\chi(x + h_{2s})
        \\ \ll q_1^{-1/2 + o(1)}q^{o(1)} + \charf{\set{h_1,\dots,h_s}\equiv\set{h_{s + 1},\dots,h_{2s}} (q_1)} 
    \end{multline*}
    (see, for example, \cite[\S B]{Ch}).
    Here, we write $\charf{\set{h_1,\dots,h_\ell}\equiv\set{h_{\ell + 1},\dots,h_{2\ell}} (q_1)}$ to be the indicator function
    that the multisets $H_1=\set{h_1\mod{q_1},\dots,h_\ell\mod{q_1}}$, $H_2 =\set{h_{\ell + 1}\mod{q_1},\dots,h_{2\ell}\mod{q_1}}$ are 
    the same. Thus, we derive that 
    \[
        \norm{\chi}_{U^s[q]}\ll q_1^{-2^{-s}}.
    \]
     
\end{proof}

We shall also require the following consequence of log-free zero density estimates.
\begin{lemma}
Let $1/2>\eta>0$ and fix a $q>0$. Suppose that $\mc S\subset\set{\chi\pmod{q}}$ is a set of nonprincipal characters 
such that all $\chi\in\mc S$, $L(s,\chi)$ has the following zero-free region: $$L(\sigma + it, \chi)\ne 0, \mathrm{ for }
|t|\le q^{10}, \sigma\ge 1 - \eta.$$ Then, for $N$ sufficiently large, $M\gg N(\log N)^{-O(1)}$, and
$\frac{\log q}{\log M}$ sufficiently small, we have 
\[
    \sum_{\chi\in\mc S}
    \bigg|\sum_{N < n\le N + M}((\Lambda(n) - q/\varphi(q))\chi(n)\bigg|
    \ll M^{1 - \eta}.
\]
\label{lem:log_free_zero_density_est}
\end{lemma}
\begin{remark}
    We have taken $X = Q\lfloor{\exp((\log Q)^{100})/Q}\rfloor$, so by the definition of dutiful moduli, we have that for all $q\le(\log X)^{1/10}$, and any Dirichlet character
$\chi\mod{(qQ)^{\floor{20s\log\log X}}}$, we have that $L(\beta, \chi)\ne 0$
for 
\[
    \beta\in \bigg[1 - \frac{1}{(\log X)^{2}}, 1\bigg].
\]
So, we will take $\eta=(\log X)^{-2}$ in the application of the above lemma.
\end{remark}
\begin{proof}
    First, we note that the contribution of the principal character can be handled 
    by the prime number theorem, so we assume $\chi$ is nonprincipal. For such $\chi$, we have that $$\frac{q}{\varphi(q)}\bigg|\sum_{n\le N}\chi(n)\bigg|\le q \ll M^{1-\eta}.$$ 
    Thus it remains to bound $$\sum_{\chi\in\mc S}\bigg|\sum_{N\leq n\leq N+M} \Lambda(n)\chi(n)\bigg|.$$
    
    For $T = q^{10}$, by Perron's formula (see Proposition 5.25 of \cite{IK}),
    we have
    \[
        \sum_{N < n\le N + M}\Lambda(n)\chi(n) 
        = \sum_{\rho_\chi}\frac{(N + M)^{\rho_\chi} - N^{\rho_\chi}}{\rho_\chi} 
        + O\bigg(\frac NT(\log Nq)^2\bigg),
    \]
    where $\rho_\chi = \sigma_\chi + i\gamma_\chi$ runs over nontrivial zeros of $L(s, \chi)$ with $|\gamma_\chi|\le T$. We have 
    \[
        \bigg|\frac{(N + M)^{\rho_\chi} - N^{\rho_\chi}}{\rho_\chi}\bigg|\ll (1 + |\gamma_\chi|)MN^{\sigma_\chi - 1},
    \]
    so that 
    \begin{align*}
        \bigg|\sum_{N < n\le N + M}\Lambda(n)\chi(n)\bigg|&\ll Mq^{10}\sum_{\chi\in\mc S}\sum_{\rho_\chi} N^{\sigma_\chi - 1}\\
        &= Mq^{10}\bigg(2N^{-\frac{1}{2}}N(1/2, T) 
        + 2\log N\int_{1/2}^{1 - \eta} N^{\sigma - 1}N(\sigma, T)d\sigma\bigg),
    \end{align*}
    where $$N(\sigma',T) = \#\{\sigma + it: L(\sigma+it,\chi)=0, |t|\leq T, \sigma\geq \sigma'\}.$$
    The desired result then follows upon using the log-free zero density estimate
    $$N(\sigma, T)\ll (qT)^{c(1 - \sigma)},$$
    where this bound accounts for a possible Siegel zero (as our zero-free region assumed will not in general rule out Siegel zeros).
\end{proof}

\begin{proof}[Proof of Proposition \ref{prop:maj_arc_main_est}]
First, let us recall that $Q=P(z)$ and write
\[
    r = \prod_{\substack{p > z\\ p^\ell||q^d}}p^\ell, \text{ and }R = \frac{(qQ)^d}{r}.
\]
Observe that both $r | q^d$ and $ Q | R$, so we have that for some $\mc P_q\subset\set{p | q : p > z}$
\[
    \frac{Q}{\varphi(Q)} = \frac{R}{\varphi(R)}\prod_{p\in\mc P_q}
    \bigg(1 - \frac{1}{p}\bigg)^{-1}
\]
Since $z\gg\log Q$, $\#\mc P_q\ll 1$, we have that
\[
    \prod_{p\in\mc P_q}\bigg(1 - \frac{1}{p}\bigg)^{-1} = 1 + O\pfrc{1}{z}. 
\]

By a similar analysis, we obtain that
\[
    \frac{Q}{\varphi(Q)} = \frac{R}{\varphi(R)}\bigg(1 + O\pfrc{1}{z}\bigg)
    = \frac{rR}{\varphi(rR)}\bigg(1 + O\pfrc{1}{z}\bigg).
\]
Therefore, after splitting into residue classes modulo $R$, it suffices to show that
\begin{equation*}
    \sum_{\substack{N < n\le N + M\\ n\equiv b(R)}}
    (\Lambda(n) - R/\varphi(R))F(\gamma(n)\Gamma)\ll
    \frac{M}{\varphi(R)}(\log Q)^{-\delta}
     \label{eq:split_up_residue_classes}
\end{equation*}
for some $\delta > 0$.
By orthogonality and the fact that $(r, R) = 1$, we have that
\begin{align*}
    E &= \sum_{\substack{N < n\le N + M\\ n\equiv b(R)}}
    (\Lambda(n) - R/\varphi(R)) F(\gamma(n)\Gamma)\\
    &= \frac{1}{\varphi(R)}\sum_{\chi_1(R)}\conj{\chi_1}(b)
    \frac{1}{\varphi(r)}\sum_{\chi_2(r)} \conj{\chi_2}(R)A(\chi_2)
    \sum_{N < n\le N + M}(\Lambda(n) - R/\varphi(R))\chi_1\chi_2(n),
\end{align*}
where we define
\[
    A(\chi_2) = \sum_{x(r)}\conj{\chi_2}(x)F(\gamma(Rx + b)\Gamma).
\]
Let $C = \log^{\frac{1}{10}} Q$.
Applying Lemma \ref{lem:periodic_nilseq_char}, for some $\delta>0$. the contribution of those $\chi_2$ with 
conductor at least $C$ and not divisible by an odd prime square is
$$\ll\frac{1}{\varphi(R)\varphi(r)}|A(\chi_2)| \bigg|\sum_{N\leq n\leq N+M} (\Lambda(n)-R/\varphi(R))\chi_1\chi_2(n)\bigg|
\ll \frac{M}{\varphi(R)}\frac{r}{\varphi(r)}C^{-\delta}.$$ 
Let
$\mc S_{\text{ld}}$ be the set of characters $\chi_1\chi_2$ such that $\chi_1\mod R,$ and $ \chi_2\mod r$ has
conductor at most $C$, or conductor divisible by a square of an odd prime. Then, we have that 
\[
    E\ll\frac{1}{\varphi(R)\varphi(r)}\sum_{\chi\in\mc S_{\text{ld}}}
    \bigg|\sum_{N < n\le N + M}(\Lambda(n) - rR/\varphi(rR))\chi(n)\bigg| 
    + \frac{Mr}{\varphi(R)\varphi(r)}C^{-\delta}.
\]
Note that since $Q$ is a dutiful modulus and $C = (\log Q)^{1/10}$ (noting that 
$R | Q^m$ for some $m\ll\log q$), the desired result follows 
from Lemma \ref{lem:log_free_zero_density_est}, where we take 
$\eta = \frac{1}{(\log X)^{1/10}}$ (in the case that the conductor is divisible by the square of an odd prime, it
can not be a quadratic character, and we still have the superior zero-free region as a result).

\end{proof}

\section{Acknowledgments}
The authors would like their advisor Peter Sarnak for his support. The authors also thank Lilian Matthiesen, Kannan Soundararajan, Terence Tao, Joni Ter\"av\"ainen, and Aled Walker for helpful comments and suggestions. The second author is supported by the National Science Foundation Graduate Research Fellowship Program under Grant No. DGE-2039656. Any opinions, findings, and conclusions or recommendations expressed in this material are those of the author(s) and do not necessarily reflect the views of the National Science Foundation.

\bibliographystyle{abbrv}
\bibliography{biblio}

\end{document}